\documentclass[11pt,twoside,reqno,psamsfonts]{amsart}

\usepackage[left=2.9cm,top=3cm,right=2.9cm]{geometry}             
\usepackage[colorlinks = true, citecolor = black, linkcolor = black, urlcolor = black, pdfstartview=FitH]{hyperref}  
\usepackage[pdftex]{graphicx}
\usepackage[utf8]{inputenc}
\usepackage{microtype, float}
\usepackage{amsmath, verbatim}
\usepackage{amssymb,mathtools}
\usepackage{epstopdf}
\usepackage{epsfig}
\usepackage{mathscinet}

\usepackage{marginfix}
\usepackage{marginnote}
\numberwithin{equation}{section}

\theoremstyle{plain}
\newtheorem{thm}{Theorem}[section]
\newtheorem{theorem}[thm]{Theorem}

\newtheorem{lemma}[thm]{Lemma}

\theoremstyle{definition}
\newtheorem{defn}[thm]{Definition}

\theoremstyle{remark}

\newcommand{\R}{\mathbb{R}}
\newcommand{\C}{\mathbb{C}}

\renewcommand{\P}{\mathbb{P}}

\renewcommand{\H}{\mathbb{H}}
\newcommand{\Z}{\mathbb{Z}}

\renewcommand{\epsilon}{\varepsilon}
\renewcommand{\rho}{\varrho}
\renewcommand{\phi}{\varphi}
\newcommand{\1}{\mathrm{\mathbf{1}}}

\newcommand{\df}{\mathrm{d}}

\renewcommand{\div}{\mathrm{div}}

\DeclareMathOperator{\Vol}{Vol}
\DeclareMathOperator{\InjRad}{InjRad}

\DeclareMathOperator{\PSL}{PSL}

\DeclareMathOperator{\grad}{grad}

\title[Delocalisation of eigenfunctions on large genus random surfaces]{Delocalisation of eigenfunctions on large genus random surfaces}

\author[Joe Thomas]{Joe Thomas}\thanks{J. Thomas was supported by the Dean's Award from the University of Manchester.}
\address{Department of Mathematics, University of Manchester, UK}
\email{joe.thomas-3@postgrad.manchester.ac.uk}


\begin{document}
\maketitle
	\begin{abstract}
		We prove that eigenfunctions of the Laplacian on a compact hyperbolic surface delocalise in terms of a geometric parameter dependent upon the number of short closed geodesics on the surface. In particular, we show that an $L^2$ normalised eigenfunction restricted to a measurable subset of the surface has squared $L^2$-norm $\varepsilon>0$, only if the set has a relatively large size -- exponential in the geometric parameter. For random surfaces with respect to the Weil-Petersson probability measure, we then show, with high probability as $g\to\infty$, that the size of the set must be at least the genus of the surface to some power dependent upon the eigenvalue and $\varepsilon$.
	\end{abstract}

\section{Introduction}
\subsection{Background}
The study of the Laplacian operator $\Delta = -\div\grad$ has been undertaken from a multitude of different perspectives. When considered as an operator on function spaces of Riemannian manifolds, a commonplace theme has been to study the connection of properties of the eigenfunctions with respect to their eigenvalue. For example, in a quantum chaotic setting, that is, where the underlying dynamics of the geodesic flow are chaotic, there is great interest in the behaviour of the probability measures $|\psi|^2\df\mathrm{Vol}_M$. Here, $\psi$ is an $L^2$-normalised eigenfunction of the Laplacian and $\df\mathrm{Vol}_M$ is the standard volume measure on $M$. In particular, if the manifold $M$ is compact, then one can consider an orthonormal basis of $L^2(M)$ consisting of Laplacian eigenfunctions $\{\psi_{\lambda_j}\}_{j\geq 0}$ with eigenvalues $0=\lambda_0<\lambda_1\leq\ldots\to\infty$. One may then consider weak-* limits of the measures $|\psi_{\lambda_j}|^2\df\mathrm{Vol}_M$ as $j\to\infty$. An overarching conjecture by Rudnick and Sarnak \cite{Ru.Sa1994}, called the Quantum Unique Ergodicity conjecture, states that when the manifold has negative sectional curvature, these measures converge to the volume measure on the space. Essentially, this is asking whether the eigenfunctions become fully delocalised in the eigenvalue aspect. In this limit, there have been several angles of approach to demonstrating delocalisation; for example through computing $L^\infty$-norm bounds of eigenfunctions, and studying the entropy of the eigenfunction measures (see \cite{An2008, An.No2007, Br.Li2014, Iw.Sa1995}).

Rather than the eigenvalue aspect, in this article we shall consider a  delocalisation result of the eigenfunctions in the large \textit{spatial} aspect on compact hyperbolic surfaces. This allows one to understand how the eigenfunctions are affected by a large volume geometry (or equivalently, a large genus by the Gauss-Bonnet Theorem). Such a perspective is commonplace in the regular graph literature, since the spectrum of the Laplacian is bounded. Moreover, the spatial aspect has been the subject of several recent results on surfaces in part due to their spectral and geometric similarities to regular graphs. In the arithmetic surface setting, it is a natural limit  to study due to its connection with the level aspect. In this setting, some delocalisation results have been obtained by Saha \cite{Sa2020} and Hu and Saha \cite{Hu.Sa2019} in terms of the sup norms of the eigenfunctions of the form
	\begin{align*}
		 \|\psi_\lambda\|_\infty \lesssim_\lambda g^{-\alpha}\|\psi_\lambda\|_2,
	\end{align*}
for some exponent $\alpha>0$, and $g$ the genus of the surface. The notion of delocalisation that we shall explore here will seek to understand if eigenfunctions can have partial, or conversely, near full concentration, on certain subsets of a compact hyperbolic surface.\par 

Before stating our results precisely, it will be insightful to discuss some recent work that has connected Laplacian eigenfunctions to the surface geometry, as a way to highlight the different geometric influences that have so far been explored. In the work of Le Masson and Sahlsten \cite{Le.Sa2017}, a spatial analogue of quantum ergodicity for compact hyperbolic surfaces was developed via a Benjamini-Schramm type of convergence. This result is analogous to similar work on regular graphs by Anantharaman and Le Masson \cite{An.Le2015}. A sequence of compact hyperbolic surfaces $(X_n)_{n\geq 0}$ Benjamini-Schramm converges to the hyperbolic plane if for all $R>0$,
	\begin{align*}
		\frac{\mathrm{Vol}(z\in X_n: \mathrm{InjRad}_{X_n}(z)<R)}{\mathrm{Vol}(X_n)} \to 0,
	\end{align*}
as $n\to\infty$ (here $\mathrm{InjRad}_{X_n}(z)$ is the injectivity radius of the surface $X_n$ at the point $z\in X_n$). Geometrically, this means that the proportion of points on the surface with small injectivity radius, or equivalently at least one short geodesic loop based at that point, is small in the limit. For a more general notion of Benjamini-Schramm convergence, see the articles \cite{Ab.Be.Bi.Ge.Ni.Ra.Sa2017,Ab.Be.Le2018}. Benjamini-Schramm convergence can be seen as an assumption on the \textit{global} geometry of the surfaces as it requires that the geometry of the surface around \textit{most} points is `well-behaved'. It turns out that this global geometric assumption is typical for a fixed surface at appropriately chosen scales of $R$, dependent upon the surface volume/genus. More precisely, Monk \cite{Mo2020} shows that for each $g\geq 2$, there exists a subset $\mathcal{M}'_g$ of the moduli space of compact genus $g$ hyperbolic surfaces such that, for any surface $X\in \mathcal{M}_g'$ one has
	\begin{align*}
		\frac{\mathrm{Vol}(z\in X: \mathrm{InjRad}_{X}(z)<\frac{1}{6}\log(g))}{\mathrm{Vol}(X)} = O\left(g^{-\frac{1}{3}}\right).
	\end{align*}
When one considers the Weil-Petersson probability measure on the moduli space of fixed genus (we will discuss this random model in more detail later), the probability of $\mathcal{M}'_g$ tends to one as $g\to\infty$. Thus, selecting a sequence of surfaces $(X_g)_g$ where $g$ is the genus of $X_g$, and $X_g\in \mathcal{M}'_g$ for each $g$, one obtains Benjamini-Schramm convergence of this sequence, and the measure of $\mathcal{M}'_g$ approaches one as $g\to\infty$. Using this condition, Monk is able to obtain information about the structure of the Laplacian spectrum for random surfaces of large genus. \par 

This Benjamini-Schramm assumption on surfaces can be contrasted with a \textit{local} geometric assumption upon the surface geometry, that has been exploited by Gilmore, Le Masson, Sahlsten and Thomas in \cite{Gi.Le.Sa.Th2021}. The focus of their work was on the $L^p$ norms of the eigenfunctions of the Laplacian, rather than the spectrum, and how they are influenced by a large surface genus. There the authors require a strong control over the local topology around \textit{every} point on the surface, not just control around a large proportion of points as is offered by the Benjamini-Schramm condition. Specifically, they require that every point on the surface is the base of only a small number of `short' primitive geodesic loops. By comparison, the Benjamini-Schramm condition roughly states that the proportion of points that are the base of at \textit{least} one `short' geodesic loop is small compared to the surface volume.\par 

 The reason why this control over all points on the surface is beneficial, is highlighted when using the Selberg pre-trace formula to infer properties about Laplacian eigenfunctions. Indeed, understanding the behaviour of an eigenfunction at a certain point with this formula requires one to look at all the geodesic loops on the surface based at that point. On the other hand, being a global property, the Benjamini-Schramm condition seems more suited to understanding properties of the \textit{spectrum} of the Laplacian. This is because one can use the Selberg trace formula (the integral of the pre-trace formula) to link the spectrum to an integral over the surface of a function evaluated at lengths of geodesic loops on the surface. Due to the presence of the integral, one only requires that the geodesic loops are well-behaved at \textit{most} points on the surface.\par 

In this article, we are dealing with properties of the eigenfunctions of the Laplacian, and again require strong control over the local topology of all points on the surface. For this reason, we will utilise the geometric condition for surfaces that was introduced by Gilmore, Le Masson, Sahlsten and Thomas, and this is written precisely in the statement of equation \eqref{assumption} below. It turns out that the length scale at which one can understand this local topology of points corresponds greatly to the strength of the results for the eigenfunctions. Indeed, this scale corresponds to the parameter $R(X)$ in equation \eqref{assumption} below, and from Theorem \ref{thm: det thm 1}, the larger that this can be taken, the more we can infer about eigenfunction concentration. Here, we will consider estimates for this scale for both deterministic surfaces, and those chosen with high probability as the genus of the surface tends to infinity, with respect to the Weil-Petersson random model. To aid in the understanding of how large the length scale can be, it is beneficial for us to utilise another geometric property, from which one can directly infer the geometric loop properties at every point. Indeed, this is the perspective taken by Monk and Thomas in \cite{Mo.Th2020} where the \textit{tangle-free parameter} of a surface is introduced, leading to more precise length scales. We introduce this parameter now.

\begin{defn}
Given $L>0$, a compact hyperbolic surface $X$ is said to $L$-tangle-free if every embedded pair of pants and one-holed torus in $X$ has total boundary length at least $2L$.
\end{defn}

Recall that a pair of pants is a hyperbolic surface of genus zero with three simple closed geodesic boundaries, and a one-holed torus is a genus 1 hyperbolic surface with a single simple closed geodesic boundary. When we consider total boundary length, we will mean the lengths of these geodesic boundaries on the subsurfaces. \par 
Although stated in terms of pants and one-holed tori (the fundamental building blocks of a hyperbolic surface), the tangle-free parameter $L$ of a surface provides understanding on the local topology of the surface around all points, as is required here. Indeed, this is highlighted in the following theorem.

\begin{thm}[{\cite[Theorem 9]{Mo.Th2020}}]
\label{thm:monkthomas}
Suppose that $X$ is an $L$-tangle-free surface, and let $z\in X$ with $\delta_z$ a geodesic loop based at $z$ of shortest length, $\ell(\delta_z)$. Then, any (not necessarily geodesic) loop $\beta$ based at $z$, whose length $\ell(\beta)$ satisfies
	\begin{align*}
		\ell(\beta)+\ell(\delta_z) < L,
	\end{align*}
is homotopic with fixed endpoints to a power of $\delta_z$.
\end{thm}

In other words, if the injectivity radius of an $L$-tangle-free surface at a point is less than $\frac{L}{2}$, then the shortest geodesic loop $\delta_z$ based at that point is unique. Furthermore, any other geodesic loop based at that point with length less than $\frac{L}{2}$, is homotopic with endpoints fixed at $z$ to a power of $\delta_z$. This means that the topology of the $\frac{L}{2}$-neighbourhood of any point on such a surface is well understood. \par 

In the next subsection, we will state precisely how this length scale corresponds to the required local geometric condition for this article. Of course, understanding this correspondence is only useful if one can obtain estimates on how large the parameter $L$ can be. Deterministically, notice that every surface is $L$-tangle-free for $L$ at least $\mathrm{InjRad}(X)$, the injectivity radius of the surface. This is because the total boundary length of any pair of pants and one-holed tori embedded in the surface is at least $6\mathrm{InjRad}(X)$, and $2\mathrm{InjRad}(X)$ respectively. For a surface chosen at random from the moduli space of genus $g$ with respect to the Weil-Petersson model, one may take $L= c\log(g)$ for any $0<c<1$, with probability tending to $1$ as $g\to\infty$. Further details of this will be discussed in Subsection \ref{subsec: random}. \par

\subsection{Deterministic Delocalisation} 
Let us now state precisely the delocalisation result that we prove in this article. The type of delocalisation that we will examine here answers the following: suppose an eigenfunction carries some $L^2$ mass on a subset of the surface, what information can be deduced about such a subset? Inspired by analogous results obtained for regular graphs in \cite{Br.Li2013, Br.Li2014, Ga.Sr2019}, we address how large such a subset can be in terms of the genus of the surface. Recall that, in this setting, the genus is an equivalent parameter to the volume by the Gauss-Bonnet Theorem.\par 

In a near fully delocalised case, an eigenfunction would assign a mass of the order $\frac{1}{\sqrt{g}}$ across the whole manifold $X$ (due to the $L^2$ normalisation), where $g$ is the genus of $X$. Thus, on a measurable subset $E\subseteq X$, one should expect to see the $L^2$ norm of the eigenfunction restricted to this set, to be of an order proportional to the size of the set itself. In other words, if $E$ were a subset such that $\|\psi_\lambda\1_E\|_2^2 = \varepsilon$, then one would expect a bound of the form
	\begin{align*}
		\Vol(E) \geq C\varepsilon g,
	\end{align*}
for some constant $C$, independent of the genus $g$ and the mass $\varepsilon$ (when $g$ is considered large enough). \par 
What is obtained in this article, is a drop in the exponent of the genus, dependent upon the eigenvalue of the eigenfunction, and the mass $\varepsilon$. This result is obtained for surfaces chosen with respect to the Weil-Petersson model with probability tending to one as $g\to\infty$. To achieve a result of this form, we begin with a lower bound on the volume of $E$ holding for all surfaces. This lower bound can be understood in terms of the $L$-tangle-free parameter above, or more generally, the parameter $R(X)$ associated with equation \eqref{assumption} below. Then, we use probabilistic estimates for these parameters to obtain bounds in terms of the genus. We will discuss the parameter $R(X)$ now, and contrast it to a similar parameter used for regular graphs.\par

\subsubsection*{Geometric Parameter}

For regular graphs, Brooks and Lindenstrauss \cite{Br.Li2013} prove that if a graph Laplacian eigenfunction has some $L^2$ mass on a subset of the vertices, then this subset is bounded below in terms of the size of the graph and the mass on the subset. The starting point for this result is the introduction of a geometric parameter that provides a length scale at which there are few distinct, non-backtracking walks between any two vertices in the graph shorter than this length. In particular, this can be deduced from bounds on the number of cycles based at a point in the graph, whose length are controlled by a similar length scale. This is analogous to the control offered by the tangle-free parameter discussed above on the geodesic loops based at any point on the surface. In fact, the exact formulation of the geometric property that we require here is a combinatorial bound on the number of geodesic paths between points on the surface. And, as is the case with graphs, this can be inferred from similar combinatorial bounds on the number of geodesic loops based at points. \par 
More precisely, we require that for a compact hyperbolic surface $X= \Gamma\backslash\H$, there exist constants $R(X)>0$ and $C(X)>0$ such that for all $\delta>0$, there exists $C_0(\delta)>0$ such that for any $z,w\in \mathbb{H}$ one has
	\begin{align}
	\label{assumption}
		|\{\gamma\in\Gamma: d(z,\gamma w)\leq r\}| \leq C(X)C_0(\delta)e^{\delta r}, \quad \text{for all }  r\leq R(X).
	\end{align}
Of course, one can always find such parameters for any surface $X$: take $R(X) = c\mathrm{InjRad}(X)$ for any $c<1$, $C(X) = 2$ and $C_0(\delta) = 1$. Indeed, for $r\leq R(X)$ there can be at most two elements in the set on the left-hand side otherwise one would obtain a geodesic loop on the surface of length shorter than $\mathrm{InjRad}(X)$. The crucial point is that the constant $R(X)$ represents the length scale up to which we can understand the local geometry about every point on the surface (this is highlighted in Lemma \ref{lem:setbound}). This means, it will be the controlling parameter for the lower bound on the volume of a set $E$ (see the statement of Theorem \ref{thm: det thm 1} for the precise relation). The strength of the theorem thus relies on one being able to take $R(X)$ as large as possible. In fact, we will wish for the parameter $R(X)$ to grow in terms of genus of the surface. In the Weil-Petersson model, surfaces can have small injectivity radius with positive probability (also see Theorem \ref{thm:mirzakhanilemma} below), and so the idea is to ensure that we can go past the injectivity radius scale for typical surfaces. \par 

Pushing past the injectivity radius scale can possibly require the combinatorial bound of equation \eqref{assumption} to have some dependence upon the surface itself, which is why we allow the parameter $C(X)$ to appear there. If one is not so careful, this can cause some problems when observing how the constant $C(X)$ manifests itself in the lower bound on the volume of the set $E$ in Theorem \ref{thm: det thm 1}. Thus, one must ensure that $C(X)$ is well understood, so that it will not ruin the obtained bounds. Using the tangle-free parameter allows us to find constants $R(X)$ and $C(X)$ that can be studied probabilistically, and that result in lower bounds on the volume of the set $E$ in terms of the genus. Indeed, we shall show in Lemma \ref{lem:setbound} that for an $L$-tangle-free surface, one may take $R(X)=\frac{L}{4}$ and $C(X)=\frac{1}{\min\{1,\mathrm{InjRad}(X)\}}$, both of which can be understood probabilistically as in Subsection \ref{subsec: random}. For these parameters, the constant $C_0(\delta)$ can be stated explicitly as in Lemma \ref{lem:setbound}, but its precise value is unimportant to our discussion here.\par 

The deterministic result that we obtain is split into two components. First, there is the case of tempered eigenfunctions that have eigenvalue in $[\frac{1}{4},\infty)$. These are dealth with by using a similar approach to Brooks and Lindenstrauss \cite{Br.Li2013} and Ganguly and Srivastava \cite{Ga.Sr2019} on regular graphs, through what can be seen as a smoothed cosine wave propagation operator. The untempered eigenfunctions, whose eigenvalues are in $(0,\frac{1}{4})$, are analysed through a rescaled ball averaging operator, and we can actually obtain a stronger delocalisation result in this case. 

\begin{thm}
\label{thm: det thm 1}
Let $\varepsilon>0$ be given, and suppose that $X$ is a compact hyperbolic surface with constants $R(X)$ and $C(X)$ satisfying condition \eqref{assumption}. Suppose that $\psi_\lambda$ is an $L^2$-normalised eigenfunction of the Laplacian on $X$ with eigenvalue $\lambda$, and suppose that $E\subseteq X$ is a measurable set for which
	\begin{align*}
		\|\psi_\lambda\1_E\|_2^2 =\varepsilon.
	\end{align*}
For $R(X)$ sufficiently large, if $\lambda \geq \frac{1}{4}$, there exists a universal constant $C>0$, and a constant $d(\lambda)>0$ for which
	\begin{align*}
		\mathrm{Vol}(E) \geq \frac{C\varepsilon}{C(X)} e^{d(\lambda) \varepsilon R(X)},
	\end{align*}
Moreover, if $\lambda \in (0,\frac{1}{4}-\sigma)$, then there exists a universal constant $C>0$ such that for $R(X)$ sufficiently large,
	\begin{align*}
		\mathrm{Vol}(E) \geq \frac{C\varepsilon}{C(X)} e^{(\frac{1}{4}+\frac{1}{2}\sqrt{\sigma})R(X)}.
	\end{align*}
\end{thm}

The constant $d(\lambda)$ above is made explicit later (see Theorem \ref{thm: deterministic result}). This result in particular shows that the eigenfunctions can not be large on a small set if (for instance) the $L$-tangle-free parameter of the surface is large compared to the eigenvalue (taking $R(X)=\frac{L}{4}$). 

\subsection{Random Surface Delocalisation}
\label{subsec: random}
Let us now discuss how one can use Theorem \ref{thm: det thm 1} to obtain a probabilistic result in terms of the genus/volume of the surface as desired. For this, we shall first describe the construction of the Weil-Petersson random surface model that we shall employ; a more detailed account can be found in \cite{Im.Ta1992, Mi2007a, Mi2013}. Note that other \textit{distinct} random surface constructions could also be considered, such as a random triangulations model by Brooks and Makover \cite{Br.Ma2004}, and a random cover model by Magee, Naud and Puder \cite{Ma.Na.Pu2020}. It would be interesting to see if similar results to those presented here could be realised in these models.\par 
Fix a genus $g\geq 2$ and let $\mathcal{T}_g$ denote the Teichm\"{u}ller space of marked genus $g$ closed Riemann surfaces up to marking equivalence. Then, there is a $(6g-6)$-dimensional real-analytic structure on $\mathcal{T}_g$ which carries a symplectic form $\omega_{\mathrm{WP}}$ called the \textit{Weil-Petersson form}. One obtains a volume form on $\mathcal{T}_g$ by taking a $(3g-3)$-fold wedge product of $\omega_{\mathrm{WP}}$ and normalising by $(3g-3)!$. In addition to this volume structure, there is a natural group acting on $\mathcal{T}_g$ called the \textit{mapping class group}, denoted by $\mathrm{MCG}_g$, which acts by changing the marking on a point in $\mathcal{T}_g$. The \textit{moduli space} of genus $g$ is then defined as the quotient by this action:
	\begin{align*}
		\mathcal{M}_g = \mathcal{T}_g/\mathrm{MCG}_g.
	\end{align*}
This space can be thought of as the collection of hyperbolic metrics that can be endowed on a genus $g$ surface, identified up to isometry. An important feature of the Weil-Petersson volume form defined on $\mathcal{T}_g$ is that it is invariant under the action of $\mathrm{MCG}_g$, and so it descends naturally to the moduli space. With respect to this measure, the moduli space has finite volume (see \cite{Bu2010} for an upper bound, and \cite{Mi.Zo2015} for more specific asymptotics of this volume for large genus). This allows one to define a probability measure on $\mathcal{M}_g$ called the \textit{Weil-Petersson probability measure}, and calculate probabilities in the natural way:
	\begin{align*}
		\P_g^{\mathrm{WP}}(A) = \frac{1}{\mathrm{Vol}(\mathcal{M}_g)}\int_{\mathcal{M}_g} \1_A(X)\df X,
	\end{align*}
where $\df X$ is used to denote the volume form. Commonly, one takes $A$ to be a collection of surfaces satisfying some desired geometric property. By using integration tools and volume estimates obtained by Mirzakhani \cite{Mi2007,Mi2007a,Mi2013}, one is able to obtain upper bounds for these probabilities as functions of the genus, and determine events that hold with high probability as $g\to\infty$.\par 
Recall that the probabilistic result that we require is an estimate for the surface dependent parameters $R(X)$ and $C(X)$ pertaining to the condition in equation \eqref{assumption}. We will use Lemma \ref{lem:setbound} to take $R(X)=\frac{L}{4}$, where $L$ is the tangle-free parameter of a surface, and $C(X) = \frac{1}{\min\{1,\mathrm{InjRad}(X)\}}$. We then have the following probabilistic estimate for $L$ from Monk and Thomas in \cite{Mo.Th2020}. 

\begin{thm}[{\cite[Theorem 4]{Mo.Th2020}}]
\label{thm: GLMST19}
For any $0<c<1$, one has 
	\begin{align*}
		\P_g^{\mathrm{WP}}\left(X\in\mathcal{M}_g : X\ \text{is } c\log(g)\text{-tangle-free}\right) = 1-O\left(\frac{\log(g)^2}{g^{1-c}}\right),
	\end{align*}
as $g\to \infty$. 
\end{thm}

For the parameter $C(X)$, it suffices to estimate the injectivity radius. The following result of Mirzakhani is sufficient for our purposes.

\begin{thm}[{\cite[Theorem 4.2]{Mi2013}}]
\label{thm:mirzakhanilemma}
For any $a>0$,
	\begin{align*}
		\P_g^{\mathrm{WP}}(X: \InjRad(X)\geq g^{-a}) = 1-O(g^{-2a}),
	\end{align*}
as $g\to\infty$.
\end{thm} 

Thus with probability tending to 1 as $g\to\infty$, $C(X)\geq g^{-a}$ for any $a>0$. By combining these probabilistic results with Theorem \ref{thm: det thm 1}, we obtain the following for large genus surfaces.

\begin{thm}
\label{thm: random result}
Let $\varepsilon>0$ be given, and suppose that $X$ is a compact hyperbolic surface with genus $g$ chosen randomly according to the Weil-Petersson probability model. Suppose further that $\lambda$ is an eigenvalue of the Laplacian on $X$, and $\psi_\lambda$ is an $L^2$-normalised eigenfunction of the Laplacian with eigenvalue $\lambda$. Let $E\subseteq X$ be a measurable set for which
	\begin{align*}
		\|\psi_\lambda\1_E\|_2^2 =\varepsilon.
	\end{align*}
Then, if $\lambda\geq\frac{1}{4}$, there exists a universal constant $C>0$ such that for any $0<c<\frac{1}{4}$, and $a>0$, one obtains
	\begin{align*}
		\mathrm{Vol}(E) \geq C\varepsilon g^{c\varepsilon d(\lambda)-a},
	\end{align*}
with $d(\lambda)$ as in Theorem \ref{thm: det thm 1}. If $\lambda\in (0,\frac{1}{4}-\sigma)$, there exists a universal constant $C>0$ such that for any $0<c<\frac{1}{4}$ and $a>0$, one obtains
	\begin{align*}
		\mathrm{Vol}(E) \geq C\varepsilon g^{c(\frac{1}{4}+\frac{1}{2}\sqrt{\sigma})-a}.
	\end{align*}
Both bounds hold with the rate
	\begin{align*}
		1-O\left(\frac{\log(g)^2}{g^{1-4c}}+g^{-2a}\right),
	\end{align*}
as $g\to\infty$.
\end{thm}

\noindent\textbf{Remark.} As noted, the exponent of the genus in the above result is governed exclusively by probabilistic estimates for $R(X)$. For our result, these followed from probabilistic estimates of the tangle-free parameter $L$. To improve the exponent using this method, one would need to show that typical surfaces can have a tangle-free parameter of the size $A\log(g)$ for $A$ large. However, in Monk and Thomas \cite{Mo.Th2020}, it is shown that no surface is more than $(4\log(g)+O(1))$-tangle-free, which would not be sufficient for this. Thus, any significant improvement to the exponent would require a new approach to estimating $R(X)$ for a typical surface.

\section{Harmonic Analysis for Hyperbolic Surfaces}
\label{sec: background}
We begin by defining our main object of study, hyperbolic surfaces, as well as outlining necessary tools from harmonic analysis that are used to obtain our results. One can find further details on these topics in Katok \cite{Ka1992}, Bergeron \cite{Be2016} and Iwaniec \cite{Iw2002}. \par 
The \textit{hyperbolic upper half-plane} will be a sufficient model of hyperbolic space for our purposes. This is defined by
	\begin{align*}
		\H = \{ z=x+iy \in \C\colon y>0\},
	\end{align*}
and is equipped with the Riemannian metric 
	\begin{align*}
		\df s^2 = \frac{\df x^2+\df y^2}{y^2},
	\end{align*}
which induces the standard Riemannian volume form
	\begin{align*}
		\df\mu = \frac{\df x\wedge \df y}{y^2}.
	\end{align*}
The set of orientation preserving isometries of $\H$ with this metric are the M\"{o}bius transformations given by
	\begin{align*}
		z \mapsto \frac{az+b}{cz+d},
	\end{align*}
for some $a,b,c,d\in\R$ with $ad-bc = 1$. They can be identified with the group $\mathrm{PSL}(2,\R)$ with the natural associated group action. Using this, one can make the identification $\H = \PSL(2,\R)/\mathrm{SO}(2)$. \par 
A convenient definition for a hyperbolic surface is then obtained through this group action. Indeed, consider a discrete subgroup $\Gamma <\mathrm{PSL}(2,\R)$ that acts freely upon $\H$. A \textit{hyperbolic surface} is a manifold quotient $X=\Gamma\backslash\H$. That is, the surface consists of points on $\H$ identified up to orbits of isometries in $\Gamma$. The Riemannian metric and volume measure are induced upon the surface in a natural manner. To each such subgroup $\Gamma$ (and hence to each surface), one may determine (non-uniquely) a fundamental domain in $\H$. Functions defined on the surface can be identified with $\Gamma$-invariant functions upon $\H$, or functions on such a fundamental domain. We will deal in this article exclusively with the case when $X$ is compact.\par 
The harmonic analysis tools that are required to show our result are given by invariant integral operators and the Selberg transform. Such operators are constructed from radial functions. These are bounded, even and measurable functions $k: (-\infty,\infty)\to \C$. They give rise to a function, which we also denote by the same symbol $k:\H\times\H\to\C$, through the correspondence
	\begin{align*}
		k(z,w) = k(d(z,w)),
	\end{align*}
where $d(z,w)$ is the hyperbolic distance between $z,w\in\H$. This function has the important property that it is invariant under the diagonal action of $\mathrm{PSL}(2,\R)$. That is, for any $\gamma\in\mathrm{PSL}(2,\R)$ and $z,w\in\H$ one has 
	\begin{align*}
		k(\gamma z,\gamma w) = k(z,w).
	\end{align*}
From this, one then \textit{formally} defines a function $k_\Gamma:X\times X\to\C$ called an \textit{automorphic kernel} by
	\begin{align*}
		k_\Gamma(z,w) = \sum_{\gamma\in\Gamma} k(z,\gamma w),
	\end{align*}
where we have defined $k_\Gamma$ as a $\Gamma$-periodic function on $\H$. For this sum to converge, one requires an appropriate decay condition on $k$; for instance
	\begin{align*}
		|k(\rho)|= O(e^{-\rho(1+\delta)}),
	\end{align*}
for some $\delta>0$ would suffice, and we assume such a condition from now on. We can then define an invariant integral operator $T_k$ on functions on $X$ through the formula
	\begin{align*}
		(T_kf)(z) = \int_D k_\Gamma(z,w)f(w)\df\mu(w),
	\end{align*}
where $D$ is a fundamental domain for $X$. The importance of operators defined in this manner is their connection to the Laplacian operator which we recall is defined in coordinates on $\H$ as
	\begin{align*}
		\Delta = -\div\grad = -y^2\left(\frac{\partial^2}{\partial x^2} + \frac{\partial^2}{\partial y^2}\right).
	\end{align*}
This operator commutes with isometries on $\H$ and so naturally passes to an operator on the hyperbolic surface $X$. Since $X$ is compact, the Laplacian has a discrete spectrum contained in $[0,\infty)$, with the $0$-eigenspace being simple and consisting of the constant functions. In addition, there exists an orthonormal basis $\{\psi_{\lambda_j}\}_{j=0}^\infty$ of Laplacian eigenfunctions for $L^2(X)$ with eigenvalues $0=\lambda_0< \lambda_1 \leq \lambda_2 \leq\ldots\to\infty$. \par 
An important observation is that any eigenfunction of the Laplacian is also an eigenfunction of an invariant integral operator $T_k$ on the surface. The eigenvalue of such an eigenfunction for $T_k$ can be determined by taking a \textit{Selberg transform} of the initial radial kernel. This is defined to be the Fourier transform 
	\begin{align*}
		\mathcal{S} (k)(r) = h(r) = \int_{-\infty}^{+\infty} e^{iru} g(u) \df u,
	\end{align*}
of the function
	\begin{align*}
		g(u) = \sqrt{2} \int_{|u|}^{+\infty}   \frac{k(\rho) \sinh \rho}{\sqrt{\cosh \rho - \cosh u}} 		\df\rho.
	\end{align*}
The spectrum is then provided from the following result.
\begin{thm}[{\cite[Theorem 3.8]{Be2016}}]

Let $X=\Gamma\backslash\H$ be a hyperbolic surface and $k \colon [0,\infty)\to\C$ a radial kernel. Suppose that $\psi_\lambda$ is an eigenfunction of the Laplacian with eigenvalue $\lambda=s_\lambda^2+\frac{1}{4}$ for $s_\lambda\in\C$.  Then $\psi_\lambda$ is an eigenfunction of the convolution operator $T_k$ with invariant kernel $k$ and
	\begin{align*}
		(T_k\psi_\lambda)(z) = \int_X k_\Gamma(d(z,w)) \psi_\lambda(w) \: \df\mu(w)=h(s_\lambda)\psi_\lambda(z),
	\end{align*}
where $h(s_\lambda)=\mathcal{S}(k)(s_\lambda)$.
\end{thm}
One refers to $s_\lambda$ in the above result as the \textit{spectral parameter} associated to $\lambda$. Through this result, and the Selberg transform, one can also reconstruct an invariant kernel operator with a specified spectrum. Indeed, given a suitable function $h$ one can take an inverse Selberg transform to obtain a radial kernel $k$ through the formulae
	\begin{align*}
		 g(u) = \frac1{2\pi}  \int_{-\infty}^{+\infty} e^{-isu} h(s) \df s,
	\end{align*}
and then
	\begin{align*}
		k(\rho) = -\frac1{\sqrt{2}\pi} \int_\rho^{+\infty} \frac{g'(u)}{\sqrt{\cosh u - \cosh \rho}} \df u.
	\end{align*}

\section{Delocalisation of Tempered Eigenfunctions on Large Genus Surfaces}
We start with the deterministic version of our result, and thus consider a fixed compact hyperbolic surface $X=\Gamma\backslash\H$. Let $D\subseteq \H$ be a fundamental domain of $X$, and $E\subseteq X$ a measurable set. Recall that $X$ satisfies the geometric assumption \eqref{assumption} with some constants $R(X)\geq\mathrm{InjRad}(X)>0$, $C(X)>0$ and $C_0(\delta)>0$. Suppose that $\{\psi_{\lambda_j}\}_{j=0}^\infty$ is an orthonormal basis for $L^2(X)$ of Laplacian eigenfunctions with corresponding eigenvalues $0=\lambda_0< \lambda_1 \leq \ldots \to\infty$. It is clear that in the case of the constant eigenfunctions corresponding to $\lambda_0$ that the delocalisation result holds, and so we will fix an eigenvalue $\lambda = \lambda_j$ for some $j\geq 1$. In particular, in this section we will further assume that the eigenfunction is tempered so that $\lambda \geq \frac{1}{4}$. Let $s_\lambda \in [0,\infty)$ be the spectral parameter associated with $\lambda$ through the equation $s_\lambda^2+\frac{1}{4}=\lambda$.

\subsection{Outline of the proof}
\label{subsec: proof outline}
The connection between the eigenfunction, the geodesic loop parameter $R(X)$ and the volume of the set $E$ is unified in the construction of a propagation operator. To exhibit this we utilise the following methodology:
	\begin{enumerate}
		\item We consider a family of operators that are to be seen as a smoothed cosine wave kernel, and recall how these operators act upon Laplacian eigenfunctions using results of \cite{Gi.Le.Sa.Th2021}. [Lemma \ref{lem: Pt norm}] 
		\item By selecting certain members of this family of operators and weighting them appropriately, we construct another family of operators now specialised to a certain fixed eigenvalue $\lambda$, as well as some secondary parameters that will later depend on the parameter $\varepsilon$. We then determine the operator norm of these operators. [Lemma \ref{lem: operator norm}]
		\item The eigenfunctions of the Laplacian are also eigenfunctions of the constructed family of operators, and so we study their eigenvalues under this operator family. This is done by showing that they can be written in terms of Fej\'{e}r kernels of certain orders, and so we obtain bounds on the eigenvalues using properties of these kernels. [Lemma \ref{lem: spectral action}] 
		\item Through studying the spectral decomposition of the restricted eigenfunction $\psi_\lambda\1_E$ over an orthonormal basis, we can relate a lower bound on the volume of $E$ to the previously obtained bounds on the eigenvalues and operator norms of the constructed family of operators. [Theorem \ref{thm: deterministic result}]
	\end{enumerate}

To begin, let us explain how one may take $R(X)=\frac{L}{4}$ and $C(X)= \frac{1}{\min\{1,\mathrm{InjRad}(X)\}}$ for an $L$-tangle-free surface $X$. This will allow one to contextualise the results in terms of $L$, and allow for Theorem \ref{thm: random result} to be deduced immediately from Theorems \ref{thm: det thm 1}, \ref{thm: GLMST19} and \ref{thm:mirzakhanilemma} .

\begin{lemma}
\label{lem:setbound}
Suppose that $X$ is an $L$-tangle-free compact hyperbolic surface. Then, for any $\delta>0$, there exists a constant $C_0(\delta)>0$ such that for all $z,w\in \H$, one has
	\begin{align*}
		|\{\gamma\in\Gamma:d(z,\gamma w)\leq r\}| \leq \frac{C_0(\delta)}{\min\{1,\mathrm{InjRad}(X)\}}e^{\delta r},\quad \text{for all}\ r\leq \frac{L}{4}.
	\end{align*}
\end{lemma}
\begin{proof}
It is clear from Theorem \ref{thm:monkthomas} that for any $r\leq \frac{L}{2}$, and any $z\in \mathbb{H}$, there is at most one non-identity primitive $\gamma\in\Gamma$ that is in the set
	\begin{align*}
		\{\gamma\in\Gamma: d(z,\gamma z)\leq r\}.
	\end{align*}
Any element in $\Gamma$ is the power of a primitive element as $X$ is compact. Moreover, if $\gamma_1\in \Gamma$ is equal to $\gamma_0^n$ for some primitive element $\gamma_0\in\Gamma$, then 
	\begin{align*}
		d(z,\gamma_0 z) \leq d(z, \gamma_1 z).
	\end{align*}
This means that if $\gamma_1$ is in this set, then the primitive element $\gamma_0$ is also. Combining these observations, the only elements in this set are $\gamma^n$ for some powers $n\in\Z$, and $\gamma\in\Gamma$ a single primitive element.\par 

To determine an upper bound on the number of elements, we use the fact that the distance $d(z,\gamma^n z)$ is at least $n$ times the translation distance of $\gamma$ (it would be precisely equal if $z$ were on the axis of $\gamma$). By definition, the translation distance is bounded below by twice the injectivity radius of the surface. Considering the identity and both the positive and negative powers of $\gamma$, we see the maximal number of elements in this set is
	\begin{align*}
		1+\left\lfloor \frac{r}{\mathrm{InjRad}(X)} \right\rfloor,
	\end{align*}
for any $z\in\H$. A bound on this set provides a bound on the cardinality of the set 
	\begin{align}
		\label{eq: set1}
		\left\{\gamma\in\Gamma:d(z,\gamma w)\leq \frac{r}{2}\right\},
	\end{align}
for $z,w\in\H$. Indeed, suppose there were at least 
	\begin{align*}
		 m=2+\left\lfloor \frac{r}{\mathrm{InjRad}(X)} \right\rfloor
	\end{align*}
non-identity elements in the set, labelled $\gamma_i$ for $1\leq i\leq m$. For each $2\leq i\leq m$ we have
	\begin{align*}
		d(\gamma_1w,(\gamma_i\gamma_1^{-1})(\gamma_1w))\leq d(\gamma_1w,z)+d(z,\gamma_iw)\leq r,
	\end{align*}
by the triangle inequality, and the fact that the $\gamma_i$ are in the set. As the $\gamma_i$ are distinct, $\gamma_i\gamma_1^{-1}$ is not the identity for any $i=2,\ldots,m$. This means that there are at least $m-1$ non-identity elements in the set
	\begin{align*}
		\{\gamma\in\Gamma: d(\gamma_1w,\gamma\gamma_1 w)\leq r\},
	\end{align*}
contradicting the previous bound. When including the identity, this means there are at most $m$ elements in the set stated in equation \eqref{eq: set1}. Notice that 
	\begin{align*}
		m \leq 2 + \left\lfloor \frac{r}{\min\{1,\mathrm{InjRad}(X)\}} \right\rfloor.
	\end{align*}
We wish to find a constant $C_0(\delta)>0$ so that 
	\begin{align*}
		m \leq \frac{C_0(\delta)}{\min\{1,\mathrm{InjRad}(X)\}}e^{\delta r},
	\end{align*}
for all $\delta>0$. We first bound $2+r$. If $r< 1$, then trivially, this is bounded by $3\leq 3e^{\delta r}$. If $r\geq 1$ then we can observe that 
	\begin{align*}
		2+r \leq 3r \leq 3e^{\frac{1}{\delta}}e^{\delta r}.
	\end{align*}
Indeed, 
	\begin{align*}
		3e^{\frac{1}{\delta}}e^{\delta r} \geq 3\left(1+ \frac{1}{\delta}\right)(1+\delta r) \geq 3r.
	\end{align*}
Hence given $\delta>0$, we set $C_0(\delta) = 3e^{\frac{1}{\delta}}$. Then,
	\begin{align*}
		m &\leq 2 + \left\lfloor \frac{r}{\min\{1,\mathrm{InjRad}(X)\}} \right\rfloor \\
		&\leq \frac{2+r}{\min\{1,\mathrm{InjRad}(X)\}} \\
		&\leq \frac{C_0(\delta)}{\min\{1,\mathrm{InjRad}(X)\}}e^{\delta r}.
	\end{align*}
\end{proof}

\subsection{Construction of a Family of Propagation Operators}
In this subsection we take the first step of defining an appropriate family of operators. These are largely based on their similarity to wave propagation operators, and are defined through the inverse Selberg transform. Indeed, define 
	\begin{align*}
		h_t(r) = \frac{\cos(rt)}{\cosh(\frac{\pi r}{2})},
	\end{align*}
for appropriate values of $r\in\C$ and $t\geq 0$. Denote by $k_t(\rho)$ the radial kernel obtained via the inverse Selberg transform of $h_t$. This defines an integral operator $P_t$ on functions of the hyperbolic plane via
	\begin{align*}
		P_tf(z)= \int_\H k_t(d(z,w))f(w)\df\mu(w).
	\end{align*}
The construction of this operator is similar to that used in Iwaniec and Sarnak \cite{Iw.Sa1995} when computing sup norm bounds for Laplacian eigenfunctions on arithmetic surfaces. Indeed, in their article they construct a propagation operator whose kernel is based on the Fourier transform of $h_t$. The exact kernel $k_t$ defined above has been studied by Brooks and Lindenstrauss in \cite{Br.Li2014}, and also by Gilmore, Le Masson, Sahlsten and Thomas in \cite{Gi.Le.Sa.Th2021}, and several important facts about the associated operator $P_t$ will be utilised here.\par 
Through use of the automorphic kernel one may consider $P_t$ as an operator on functions of the surface $X$. That is, we consider $P_t$ on such functions acting by
	\begin{align*}
		P_tf(z) = \int_D \sum_{\gamma\in \Gamma} k_t(d(z,\gamma w))f(w)\df\mu(w),
	\end{align*}
with $D$ a fundamental domain of $X$ as before. Let $\Pi$ denote the projection operator to the subspace orthogonal to constants defined by
	\begin{align*}
		\Pi f(z) = f(z) - \frac{1}{\sqrt{\mathrm{Vol}(X)}}\int_D f(w)\df\mu(w).
	\end{align*}
Then \cite{Gi.Le.Sa.Th2021} shows that the operators $P_t\Pi$ are bounded linear operators from $L^q(X)\to L^p(X)$ for $1\leq q\leq 2\leq p\leq\infty$ conjugate indices, when $t$ is not too large. In fact, an explicit upper bound is obtained on the operator norm. Here it will suffice to consider only the $L^1(X)\to L^\infty(X)$ norm estimates, and we replicate the statement of these bounds for the reader's convenience.

\begin{lemma}[{\cite[Lemma 3.3, Theorem 4.3]{Gi.Le.Sa.Th2021}}]
\label{lem: Pt norm}
Suppose that $X$ is a compact hyperbolic surface with $R(X)$ the parameter of condition \eqref{assumption}, and with associated constants $C(X)$ and $C_0(\delta)$. Then for $t\leq \frac{R(X)}{4}$ and any $\delta>0$,
	\begin{align*}
		\|P_t\Pi\|_{L^1(X)\to L^\infty(X)} \leq C(X)C_0(\delta)e^{-(\frac{1}{2}-\delta)t},
	\end{align*}
\end{lemma}

The proof of this result relies on estimating the expression
	\begin{align*}
		\sum_{\gamma\in\Gamma} |k_t(d(z,\gamma w))|,
	\end{align*}
for $z,w\in D$, which arises from the automorphic kernel of $P_t\Pi$. Outside of a ball of radius $4t$, Brooks and Lindenstrauss show that the kernel $k_t$ satisfies some strong exponential decay. Inside the ball of radius $4t$, we are considering points $z,w\in D$ for which $d(z,\gamma w)\leq R(X)$. The number of these terms is bounded by a sub-exponential growth from condition \eqref{assumption} which is off-set by a slight exponential decay of the kernel. Crucially, this is where control over the geodesics between all points on the surface with lengths up to the scale $R$ is utilised, and results in an exponential decay for the operator norm of $P_t\Pi$.\par 

Armed with this upper bound for the operator norm of $P_t\Pi$, we wish to construct a new operator that is specialised to the eigenvalue $\lambda$, and two auxiliary parameters that will later depend upon $\varepsilon$. This is done by taking a certain linear combination of members of the $P_t\Pi$ family for select values of $t$ at which the above bounds hold. The choice of $t$ is delicate. On one hand, we need to take enough operators in the linear combination so that the operator has an appropriate spectral action on Laplacian eigenfunctions. On the other hand, taking too many operators in the linear combination will inflate the operator norm too much. To this end, we follow the approach of \cite{Ga.Sr2019} used for regular graphs which refines and improves upon the original techniques and bounds obtained in \cite{Br.Li2013, Br.Li2014}. \par 

Recall that the Fej\'{e}r kernel of order $N$ is defined by 
	\begin{align*}
		F_N(s) = \frac{1}{N}\frac{\sin^2(Ns/2)}{\sin^2{s/2}} = 1+\sum_{j=1}^N \frac{N-j}{N/2} \cos(js).
	\end{align*}
By dividing the summation in the right hand side by certain hyperbolic cosines, one recovers a summation of functions similar to the $h_t$ defined above. We will exploit this observation to understand the spectral action of a certain linear combination of $P_t\Pi$ as a function of the Fej\'{e}r kernel. To this end, for positive integers $N$ and $r$, define 
	\begin{align*}
		W_{\lambda,r,N} = \sum_{j=1}^N \frac{N-j}{N}(\cos(r s_\lambda j)+1)P_{jr}\Pi.
	\end{align*}
With control over the values of $N$ and $r$, we can utilise the upper bound on the operator norm of $P_t\Pi$ to see that this is a bounded operator from $L^1(X)\to L^\infty(X)$, and obtain explicit bounds on the operator norm.

\begin{lemma}
\label{lem: operator norm}
Suppose that $\lambda\geq\frac{1}{4}$ is an eigenvalue of the Laplacian on a compact hyperbolic surface $X$ with parameter $R=R(X)$ from condition \eqref{assumption}, and associated constants $C(X)$ and $C_0(\delta)$. Given positive integers $N$ and $r$ satisfying $Nr\leq \frac{1}{4}R$, there exists a universal constant $\delta_0>0$ such that for all $\delta<\delta_0$, the operator $W_{\lambda,r,N}:L^1(X)\to L^\infty(X)$ is a bounded linear operator with norm
	\begin{align*}
		\|W_{\lambda,r,N}\|_{L^1(X)\to L^\infty(X)} \leq C(X)A(\delta) e^{-(\frac{1}{2}-\delta) r},
	\end{align*}
for some constant $A(\delta)>0$ dependent only upon $\delta$.
\end{lemma}
\begin{proof}
From the conditions on $N$ and $r$, we have $jr\leq \frac{1}{4}R$ for each $j=1,\ldots, N$. Utilising Lemma \ref{lem: Pt norm} we obtain
	\begin{align*}
		\|W_{\lambda,r,N}\|_{L^1(X)\to L^\infty(X)} &\leq \sum_{j=1}^N \left|\frac{N-j}{N}(\cos(r s_\lambda j)+1)\right|\|P_{jr}\Pi\|_{L^1(X)\to L^\infty(X)}\\
		&\leq 2C(X)C_0(\delta)\sum_{j=1}^N e^{-(\frac{1}{2}-\delta) jr}\\
		&\leq 2C(X)C_0(\delta)\left(e^{-(\frac{1}{2}-\delta)r}+\sum_{j=2}^\infty  e^{-(\frac{1}{2}-\delta) jr}\right)\\
		&= 2C(X)C_0(\delta)\left(e^{-(\frac{1}{2}-\delta)r}+\frac{e^{-(\frac{1}{2}-\delta)r}}{e^{(\frac{1}{2}-\delta)r}-1}\right).
	\end{align*}
Set $\delta_0 = 0.01$ so that if $\delta < \delta_0$, then $2\sinh\left(\frac{1}{2}-\delta\right)	= e^{\left(\frac{1}{2}-\delta\right)}-e^{-\left(\frac{1}{2}-\delta\right)}\geq 1$. Under this condition, since $r\geq 1$, we obtain 
	\begin{align*}
		e^{-\left(\frac{1}{2}-\delta\right)r}\left(1+e^{\left(\frac{1}{2}-\delta\right)}\right) \leq e^{-\left(\frac{1}{2}-\delta\right)}\left(1+e^{\left(\frac{1}{2}-\delta\right)}\right) \leq e^{\left(\frac{1}{2}-\delta\right)}.
	\end{align*}
This is equivalent to 
	\begin{align*}
		e^{-\left(\frac{1}{2}-\delta\right)r} \leq e^{\left(\frac{1}{2}-\delta\right)}e^{-\left(\frac{1}{2}-\delta\right)r}\left(e^{\left(\frac{1}{2}-\delta\right)r}-1\right).
	\end{align*}
Plugging this estimate into the bounds for the operator norm above then gives
	\begin{align*}
		\|W_{\lambda,r,N}\|_{L^1(X)\to L^\infty(X)} \leq 2C(X)C_0(\delta)\left(1+e^{\left(\frac{1}{2}-\delta\right)}\right)e^{-\left(\frac{1}{2}-\delta\right)r}.
	\end{align*}
Setting $A(\delta) = 2C_0(\delta)\left(1+e^{\left(\frac{1}{2}-\delta\right)}\right)$ then gives the result.
\end{proof}

\subsection{Determining the Spectral Action and Proof of Theorem \ref{thm: det thm 1} for Tempered Eigenfunctions}
We now analyse the spectrum of the operator $W_{\lambda,r,N}$ defined above. We do this by testing it against the orthonormal basis of Laplacian eigenfunctions considered at the start of this section.

\begin{lemma}
\label{lem: spectral action}
Suppose that $\lambda\geq\frac{1}{4}$ and $\mu\in [0,\infty)$ are eigenvalues of the Laplacian on a compact hyperbolic surface $X$ with geodesic loops parameter $R=R(X)$ given by \eqref{assumption}. Fix positive integers $N$ and $r$ satisfying $Nr\leq \frac{1}{4}R$. If $\psi_\mu$ is an eigenfunction of the Laplacian on $X$ with eigenvalue $\mu$, then $\psi_\mu$ is an eigenfunction of the operator $W_{\lambda,r,N}$, and the following bounds on its eigenvalue hold.
	\begin{enumerate}
		\item If $\mu \geq \frac{1}{4}$, then the eigenvalue of $\psi_\mu$ under the action of $W_{\lambda,r,N}$ is at least $-1$.
		\item If $\mu \in [0,\frac{1}{4})$, then the eigenvalue of $\psi_\mu$ under the action of $W_{\lambda,r,N}$ is at least $0$.
		\item The eigenvalue of $\psi_\lambda$ under $W_{\lambda,r,N}$ is at least $\frac{N-4}{4\cosh(s_\lambda \pi/2)}$.
	\end{enumerate}
\end{lemma}
\begin{proof}
The fact that $\psi_\mu$ is an eigenfunction of $W_{\lambda,r,N}$ is immediate from the construction of the operator as a linear combination of $P_t\Pi$ for various values of $t$. To analyse the eigenvalue of $\psi_\mu$, we will rewrite it as a function of Fej\'{e}r kernels. If $\mu = 0$, then it is obvious from the definition of $\Pi$ that the eigenvalue is zero, so assume that $\mu>0$. Then,
	\begin{align*}
		W_{\lambda,r,N}\psi_\mu &=  \sum_{j=1}^N \frac{N-j}{N}(\cos(r s_\lambda j)+1)\frac{\cos(rs_\mu j)}{\cosh\left(\frac{s_\mu\pi}{2}\right)}\psi_\mu.
	\end{align*}
For small eigenvalues $\mu\in (0,\frac{1}{4})$, it is easy to see that the summation is non-negative. Indeed, $s_\mu$ will be purely imaginary and lie in $(0,\frac{1}{2})i$, so that $\frac{s_\mu\pi}{2} \in (0,\frac{\pi}{4})i$ and each term in the summation is non-negative. 
To deal with values of $\mu$ at least $\frac{1}{4}$, we rewrite the above eigenvalue by splitting the summation. Notice that
	\begin{align*}
		\frac{1}{\cosh\left(\frac{s_\mu\pi}{2}\right)}&\sum_{j=1}^N \frac{N-j}{N}\cos(r s_\lambda j)\cos(r s_\mu j)\\
		&\hspace{-1cm}= \frac{1}{2\cosh\left(\frac{s_\mu\pi}{2}\right)}\sum_{j=1}^N \frac{N-j}{N}(\cos(jr (s_\lambda+s_\mu))+\cos(jr (s_\lambda-s_\mu)))\\
		&\hspace{-1cm}= \frac{1}{4\cosh\left(\frac{s_\mu\pi}{2}\right)}\left(1+2\sum_{j=1}^N \frac{N-j}{N}\cos(jr (s_\lambda+s_\mu))+1+2\sum_{j=1}^N \frac{N-j}{N}\cos(jr (s_\lambda-s_\mu))-2\right) \\
		&\hspace{-1cm}= \frac{1}{4\cosh\left(\frac{s_\mu\pi}{2}\right)}(F_N(r(s_\lambda+s_\mu))+F_N(r(s_\lambda-s_\mu))-2).
	\end{align*}
Similarly, we have 
	\begin{align*}
		\frac{1}{\cosh\left(\frac{s_\mu\pi}{2}\right)}&\sum_{j=1}^N \frac{N-j}{N}\cos(r s_\mu j)=\frac{1}{2\cosh\left(\frac{s_\mu\pi}{2}\right)}(F_N(rs_\mu)-1).
	\end{align*}
The eigenvalue can then be analysed by using properties of the Fej\'{e}r kernel. Indeed, we have that $F_N(s)\geq 0$ from the sine representation of the Fej\'{e}r kernel for all $s\in\R$. Thus, the eigenvalue is bounded below by 
	\begin{align*}
		\frac{1}{4\cosh\left(\frac{s_\mu\pi}{2}\right)}(0+0-2)+\frac{1}{2\cosh\left(\frac{s_\mu\pi}{2}\right)}(0-1) = -\frac{1}{\cosh\left(\frac{s_\mu\pi}{2}\right)} \geq -1.
	\end{align*}
For $\mu=\lambda$ we note that $F_N(0)=N$ so that a lower bound is given by 
	\begin{align*}
		\frac{1}{4\cosh\left(\frac{s_\lambda\pi}{2}\right)}(N+0-2)+\frac{1}{2\cosh\left(\frac{s_\lambda\pi}{2}\right)}(0-1) = \frac{N-4}{4\cosh\left(\frac{s_\lambda\pi}{2}\right)},
	\end{align*}
as required.
\end{proof}

Understanding the bounds on the spectrum of $W_{\lambda,r,N}$ allows one to obtain inequalities involving the matrix coefficients of certain functions under the operator. This is crucial since we will examine the action of $W_{\lambda,r,N}$ upon $\psi_\lambda\1_E$ via a decomposition over an orthonormal basis of eigenfunctions for $L^2(X)$. In fact, by manipulation of norms, we will see that the lower bounds on eigenvalues from Lemma \ref{lem: spectral action}, along with the upper bound for the operator norm in Lemma \ref{lem: operator norm}, will be sufficient to obtain a lower bound on the set volume.

\begin{theorem}
\label{thm: deterministic result}
Fix $\varepsilon>0$ and suppose that $X$ is a compact hyperbolic surface. Let $\psi_\lambda$ be an $L^2$-normalised eigenfunction of the Laplacian on $X$ with eigenvalue $\lambda\geq\frac{1}{4}$, and suppose that $E\subseteq X$ is a measurable set for which
	\begin{align*}
		\|\psi_\lambda\1_E\|_2^2 =\varepsilon.
	\end{align*}
Suppose that $R=R(X)$ and $C(X)$ are parameters for the surface $X$ from \eqref{assumption} with $R$ sufficiently large. Then, there exists some universal constant $C>0$ for which
	\begin{align*}
		\mathrm{Vol}(E) \geq \frac{C\varepsilon}{C(X)} e^{d(\lambda) \varepsilon R},
	\end{align*}
where $d(\lambda)$ can be taken to be 
	\begin{align*}
		d(\lambda) = \frac{1}{256\cosh\left(\frac{s_\lambda \pi}{2}\right)}.
	\end{align*}
\end{theorem}
\begin{proof}
Set the parameters $r$ and $N$ as follows:
	\begin{align*}
		N &= \left\lfloor 8\varepsilon^{-1}\cosh\left(\frac{s_\lambda \pi}{2}\right)\right\rfloor,\\
		r &= \left\lceil \frac{1}{8}N^{-1}R \right\rceil.
	\end{align*}
Since $\psi_\lambda$ is $L^2$-normalised, the parameter $\varepsilon$ is bounded above by $1$. Thus $N \geq 1$, and both $N$ and $r$ are positive integers. Additionally, 
	\begin{align*}
		rN \leq \frac{1}{8}R + N \leq \frac{1}{4}R,
	\end{align*}
when $R \geq 64\varepsilon^{-1}\cosh\left(\frac{s_\lambda\pi}{2}\right)$. Thus, for sufficiently large $R$, the parameters $N$ and $r$ satisfy the hypotheses of Lemmas \ref{lem: operator norm} and \ref{lem: spectral action}. Fix $\delta>0$ less than the universal constant $\delta_0$ from Lemma \ref{lem: operator norm}. \par 

By use of the H\"{o}lder and Cauchy-Schwarz inequalities,
	\begin{align*}
		|\langle W_{\lambda,r,N}(\psi_\lambda\1_E),\psi_\lambda\1_E\rangle|&\leq \|W_{\lambda,r,N}(\psi_\lambda\1_E)\overline{\psi_\lambda\1_E}\|_1\\
		&\leq \|W_{\lambda,r,N}(\psi_\lambda\1_E)\|_\infty\|\psi_\lambda\1_E\|_1\\
		&\leq \|W_{\lambda,r,N}\|_{L^1\to L^\infty}\|\psi_\lambda\1_E\|_1^2\\
		&\leq \|W_{\lambda,r,N}\|_{L^1\to L^\infty}\Vol(E)\|\psi_\lambda\1_E\|_2^2\\
		&= \|W_{\lambda,r,N}\|_{L^1\to L^\infty}\Vol(E)\varepsilon.
	\end{align*}
Applying the operator norm bound of Lemma \ref{lem: operator norm} then gives
	\begin{align*}
		|\langle W_{\lambda,r,N}(\psi_\lambda\1_E),\psi_\lambda\1_E\rangle|&\leq C(X)A(\delta)e^{-(\frac{1}{2}-\delta)r}\Vol(E)\varepsilon.
	\end{align*}
We now seek a lower bound on this same inner product. We do this by considering the action of the operator $W_{\lambda,r,N}$ on the spectral decomposition of $\psi_\lambda\1_E$ over the orthonormal basis. Indeed, write
	\begin{align*}
		\psi_\lambda\1_E = \langle \psi_\lambda\1_E,\psi_\lambda\rangle\psi_\lambda + f_{\mathrm{temp}} + f_{\mathrm{untemp}},
	\end{align*}
where $f_{\mathrm{temp}}$ corresponds to the tempered part of the spectrum with the term corresponding to $\psi_\lambda$ removed, and $f_{\mathrm{untemp}}$ corresponds to the untempered part of the spectrum. From Lemma \ref{lem: spectral action}, we known the action of $W_{\lambda,r,N}$ on each of these pieces of the decomposition and thus, 
	\begin{align*}
		 \langle W_{\lambda,r,N}(\langle\psi_\lambda \1_E,\psi_\lambda\rangle\psi_\lambda),\langle\psi_\lambda \1_E,\psi_\lambda\rangle\psi_\lambda\rangle &\geq \varepsilon^{-1}\|\psi_\lambda\|_2^2|\langle\psi_\lambda \1_E,\psi_\lambda\rangle|^2=\varepsilon^{-1}|\langle\psi_\lambda \1_E,\psi_\lambda\rangle|^2,\\
		 \langle W_{\lambda,r,N}(f_{\mathrm{temp}}),f_{\mathrm{temp}}\rangle &\geq - \|f_{\mathrm{temp}}\|_2^2,\\
		 \langle W_{\lambda,r,N}(f_{\mathrm{untemp}}),f_{\mathrm{untemp}}\rangle &\geq 0.
	\end{align*}
By using orthogonality and these inequalities, we see that
	\begin{align}
	\label{eqn:lowerbound}
		\langle W_{\lambda,r,N}(\psi_\lambda\1_E),\psi_\lambda\1_E\rangle &\geq \varepsilon^{-1}|\langle\psi_\lambda\1_E,\psi_\lambda\rangle|^2-\|f_{\mathrm{temp}}\|_2^2.
	\end{align}
Now, notice that 
	\begin{align*}
		|\langle\psi_\lambda\1_E,\psi_\lambda\rangle| = \|\psi_\lambda\1_E\|_2^2 > \varepsilon,
	\end{align*}
and also by an application of Pythagoras' theorem that
	\begin{align*}
		\|f_{\mathrm{temp}}\|_2^2 &\leq \|\psi_\lambda\1_E\|_2^2 - |\langle\psi_\lambda\1_E,\psi_\lambda\rangle|^2\\
		&= \|\psi_\lambda\1_E\|_2^2(1-\|\psi_\lambda\1_E\|_2^2)\\
		&\leq \|\psi_\lambda\1_E\|_2^2 (1-\varepsilon).
	\end{align*}
Putting these into the lower bound of equation \eqref{eqn:lowerbound} gives
	\begin{align*}
		\langle W_{\lambda,r,N}(\psi_\lambda\1_E),\psi_\lambda\1_E\rangle \geq \|\psi_\lambda\1_E\|_2^2(1-(1-\varepsilon)) \geq 
		 \varepsilon^2.
	\end{align*}
Combining the upper and lower bounds on the inner product then provides
	\begin{align*}
		\mathrm{Vol}(E) \geq \frac{C\varepsilon}{C(X)}e^{(\frac{1}{2}-\delta)r},
	\end{align*}
where $C= \frac{1}{A(\delta)}$, and the $\delta$ dependence is suppressed since it is fixed. We now compute using the assigned values of $\delta, r$ and $N$ that
	\begin{align*}
		\left(\frac{1}{2}-\delta\right)r \geq \frac{1}{32}N^{-1}R \geq \frac{\varepsilon R}{256\cosh\left(\frac{s_\lambda\pi}{2}\right)}.
	\end{align*}
This concludes the proof with 
	\begin{align*}
		d(\lambda) = \frac{1}{256\cosh\left(\frac{s_\lambda\pi}{2}\right)}.
	\end{align*}
\end{proof}

\section{Delocalisation of Untempered Eigenfunctions on Large Surfaces}
We now turn to studying the eigenfunctions corresponding to small eigenvalues. As before, let $X=\Gamma\backslash\H$ be a compact hyperbolic surface with associated fundamental domain $D\subseteq \H$, and let $E\subseteq X$ be a measurable subset. We will suppose this time that $\psi_\lambda$ is an eigenfunction of the Laplacian with eigenvalue $\lambda$ contained in the interval $(0,\frac{1}{4}-\sigma)$ for some $\sigma>0$. This in particular means that the spectral parameter $s_\lambda$ for the eigenvalue is contained in the set $(\sqrt{\sigma},\frac{1}{2})i$. \par 
The methodology for bounding the volume of $E$ will follow the same steps as in the tempered case. In fact, one can obtain an identical lower bound on the volume by using the work of the previous section, along with the operator
	\begin{align*}
		W_{\lambda,r,N} := W_{\frac{1}{4},r,N},
	\end{align*}
for $\lambda\in (0,\frac{1}{4})$. We instead opt here to use a different operator which allows us to obtain a stronger delocalisation result, by removing the $\varepsilon$ dependence from the exponent of the volume lower bound. \par 

The operator we use will be a rescaled ball-averaging operator on the surface. The kernel of this operator is given by
	\begin{align*}
		k_{t,\lambda}(\rho) = \frac{\1_{\{\rho\leq t\}}(\rho)}{\cosh(t)^{\frac{1}{2}(1+\sqrt{\sigma})}}.
	\end{align*}
In the usual way, we obtain an operator acting on functions on the surface through the following formula:
	\begin{align*} 
		B_{t,\lambda}f(z) = \frac{1}{\cosh(t)^{\frac{1}{2}(1+\sqrt{\sigma})}}\int_D \sum_{\gamma\in \Gamma} \1_{\{d(z,\gamma w)\leq t\}}(w)f(w)\df\mu(w).
	\end{align*}
The $L^1(X)\to L^\infty(X)$ operator norm of $B_{t,\lambda}$ is then bounded by
	\begin{align*}
		\sup_{z,w\in D}\frac{1}{\cosh(t)^{\frac{1}{2}(1+\sqrt{\sigma})}}\sum_{\gamma\in \Gamma} |\1_{\{d(z,\gamma w)\leq t\}}(w)|.
	\end{align*}
Suppose that $R(X), C(X)$ and $C_0(\delta)$ are parameters associated to the surface $X$ through condition \eqref{assumption}. For $t\leq R(X)$, and fixed $z,w\in D$, the number of terms in the summand is bounded by $C(X)C_0(\delta)e^{\delta t}$. The $L^1(X)\to L^\infty(X)$ operator norm of $B_{t,\lambda}$ for $t\leq R(X)$ is thus bounded by
	\begin{align}
	\label{eq: B_R,lambda norm}
		\|B_{t,\lambda}\|_{L^1(X)\to L^\infty(X)} \leq C(X)C_0(\delta) e^{\frac{1}{2}(2\delta -1-\sqrt{\sigma})t},
	\end{align}
for any $\delta>0$. This bound serves the same function as Lemma \ref{lem: operator norm} from the previous section. The second key ingredient required is an analogue of Lemma \ref{lem: spectral action} to understand the spectral action of $B_{t,\lambda}$. Recall that the spectral action of operators defined through a kernel in this way, is obtained from the Selberg transform of the kernel. This can be computed as follows.

\begin{lemma}
The Selberg transform of the function $k_{t,\lambda}(\rho)$ is given by
	\begin{align*}
		h_{t,\lambda}(r) = \frac{4\sqrt{2}}{\cosh(t)^{\frac{1}{2}\sqrt{\sigma}}}\int_0^t \cos(ru)\sqrt{1-\frac{\cosh(u)}{\cosh(t)}}\df u.
	\end{align*}
\end{lemma}

\begin{proof}
We use the formulae quoted in Section \ref{sec: background} to determine the Selberg transform. Firstly notice that
	\begin{align*}
		g(u)&= \frac{\sqrt{2}}{\cosh(t)^{\frac{1}{2}(1+\sqrt{\sigma})}}\int_{|u|}^t \frac{\sinh(\rho)}{\sqrt{\cosh(\rho)-\cosh(u)}}\df \rho \1_{\{|u|\leq t\}}(u)\\
		&=\frac{2\sqrt{2}}{\cosh(t)^{\frac{1}{2}(1+\sqrt{\sigma})}}\sqrt{\cosh(t)-\cosh(u)} \1_{\{|u|\leq t\}}(u).
	\end{align*}
Thus, one obtains
	\begin{align*}
		h_{t,\lambda}(r) &= \frac{2\sqrt{2}}{\cosh(t)^{\frac{1}{2}\sqrt{\sigma}}}\int_{-t}^t e^{iru}\sqrt{1-\frac{\cosh(u)}{\cosh(t)}}\df u\\
		&= \frac{4\sqrt{2}}{\cosh(t)^{\frac{1}{2}\sqrt{\sigma}}}\int_{0}^t \cos(ru)\sqrt{1-\frac{\cosh(u)}{\cosh(t)}}\df u.
	\end{align*}
\end{proof}

From now on, we will work with the operator $B_{t,\lambda}$ for $t=R(X)$. To obtain the desired bounds on the spectral action, we will require the following lemma that we isolate for readability. The result is a purely technical calculation, and so the reader who wishes to follow the main line of argument for the volume bounds will be at no loss by skipping over the proof.

\begin{lemma}
\label{lem: technical lemma}
Suppose that $a \in (\sqrt{\sigma},\frac{1}{2})$ for some $\sigma >0$, then for all $R\geq 2$, 
	\begin{align*}
		\int_0^R \cosh(au)\sqrt{1-\frac{\cosh(u)}{\cosh(R)}}\df u \geq \frac{1}{3}\sinh(\sqrt{\sigma}R).
	\end{align*}
\end{lemma}

\begin{proof}
Since $\cosh(au)$ is an increasing function in $u$, the integrand is non-negative and the expression under the square root is contained in $[0,1]$, we may bound the integral as follows:
	\begin{align*}
		\int_0^R\cosh(au)\sqrt{1-\frac{\cosh(u)}{\cosh(R)}}\df u &\geq \int_0^R\cosh(\sqrt{\sigma} u)\left(1-\frac{\cosh(u)}{\cosh(R)}\right)\df u \\
		&= \frac{\sinh(R\sqrt{\sigma})}{\sqrt{\sigma}}-\frac{1}{2}\left(\frac{\sinh((\sqrt{\sigma}+1)R)}{(\sqrt{\sigma}+1)\cosh(R)}+\frac{\sinh((1-\sqrt{\sigma})R)}{(1-\sqrt{\sigma})\cosh(R)}\right).
	\end{align*}
	This expression is then equal to
	\begin{align}
	\label{eq:boundingexp}
	\frac{2\sinh(R\sqrt{\sigma})\cosh(R)(1-\sigma)-\sinh((\sqrt{\sigma}+1)R)(\sqrt{\sigma}-\sigma)-\sinh((1-\sqrt{\sigma})R)(\sigma+\sqrt{\sigma})}{2\sqrt{\sigma}(\sqrt{\sigma}+1)(1-\sqrt{\sigma})\cosh(R)}.
	\end{align}
Since $\sqrt{\sigma}\leq\frac{1}{2}$, the denominator is bounded above by $\frac{3}{2}\cosh(R)$, and so we seek a lower bound on the numerator. Using angle sum formulae for the hyperbolic functions we see that 
	\begin{align*}
		-\sqrt{\sigma}(\sinh((\sqrt{\sigma}+1)R)+\sinh((1-\sqrt{\sigma})R)) &= -2\sqrt{\sigma}\sinh(R)\cosh(R\sqrt{\sigma})\\
		\sigma(\sinh((\sqrt{\sigma}+1)R)-\sinh((1-\sqrt{\sigma})R)) &= 2\sigma\sinh(R\sqrt{\sigma})\cosh(R).
	\end{align*}
The numerator of equation \eqref{eq:boundingexp} is thus equal to
	\begin{align*}
		2\sinh(R\sqrt{\sigma})\cosh(R)-2\sqrt{\sigma}\sinh(R)\cosh(R\sqrt{\sigma}).
	\end{align*}
Which again, using angle sum formulae for the hyperbolic functions, reduces to 
	\begin{align*}
		(2-2\sqrt{\sigma})\sinh(R\sqrt{\sigma})\cosh(R) - 2\sqrt{\sigma}\sinh((1-\sqrt{\sigma})R).
	\end{align*}
Using this lower bound on the numerator, and the upper bound on the denominator shown above, we see that \eqref{eq:boundingexp} is bounded below by 
	\begin{align*}
		\frac{2}{3}(2-2\sqrt{\sigma})\sinh(R\sqrt{\sigma})-\frac{4\sqrt{\sigma}}{3}\frac{\sinh((1-\sqrt{\sigma})R)}{\cosh(R)}.
	\end{align*}
\textit{Claim.} For $R\geq 2$, 
	\begin{align*}
	\frac{4\sqrt{\sigma}}{3}\frac{\sinh((1-\sqrt{\sigma})R)}{\cosh(R)} \leq \frac{1}{3}(2-2\sqrt{\sigma})\sinh(R\sqrt{\sigma}).
	\end{align*}
If true, this shows that the integral is bounded below by 
	\begin{align*}
		\frac{1}{3}(2-2\sqrt{\sigma})\sinh(R\sqrt{\sigma})\geq \frac{1}{3}\sinh(R\sqrt{\sigma}),
	\end{align*}
whenever $R\geq 2$, and thus the result follows.\par 
\noindent \textit{Proof of Claim.} By using an angle sum formula expansion of $\sinh((1-\sqrt{\sigma})R)$, and rearranging the inequality, we see that it suffices to show that
	\begin{align*}
		1 \leq \left(\frac{2+2\sqrt{\sigma}}{4\sqrt{\sigma}}\right)\tanh(R\sqrt{\sigma})\coth(R).
	\end{align*}
For fixed $R\geq 2$, consider the function 
	\begin{align*}
		x\mapsto \left(\frac{2+2x}{4x}\right)\tanh(Rx),
	\end{align*}
defined for $x\in (0,\frac{1}{2}]$. By differentiating, one can see that this function has a single stationary point in the interval $(0,\frac{1}{2}]$ for $R\geq 2$. Moreover, the function is strictly increasing in a neighbourhood of zero, and strictly decreasing in a neighbourhood of $\frac{1}{2}$. Thus, the stationary point is a local maxima. The values taken by this function in the domain $(0,\frac{1}{2}]$ are then bounded below by the values taken at $x=\frac{1}{2}$ and as $x\to 0^+$; these are $\frac{3}{2}\tanh(\frac{1}{2}R)$ and $\frac{1}{2}R$ respectively. In either case, we can conclude that 
	\begin{align*}
		\left(\frac{2+2\sqrt{\sigma}}{4\sqrt{\sigma}}\right)\tanh(R\sqrt{\sigma})\coth(R) \geq 1,
	\end{align*}
when $R\geq 2$, as required. 
\end{proof}

\begin{lemma}
\label{lem: untemp spectral action}
Suppose that $\lambda\in (0,\frac{1}{4}-\sigma)$ and $\mu\in [0,\infty)$ are eigenvalues of the Laplacian on a compact hyperbolic surface $X$. Suppose that $R=R(X)$ is the parameter for the surface $X$ from condition \eqref{assumption}. If $\psi_\mu$ is an eigenfunction of the Laplacian on $X$ with eigenvalue $\mu$ then the following bounds hold for $R$ sufficiently large (dependent only upon $\sigma$ and $\varepsilon$). 
	\begin{enumerate}
		\item If $\mu\geq\frac{1}{4}$, then the eigenvalue of $\psi_\mu$ under the action of $B_{R,\lambda}$ is at least $-1$.
		\item If $\mu\in[0,\frac{1}{4})$, then the eigenvalue of $\psi_\mu$ under the action of $B_{R,\lambda}$ is at least $0$.
		\item The eigenvalue of $\psi_\lambda$ under $B_{R,\lambda}$ is at least $\varepsilon^{-1}$.
	\end{enumerate}
\end{lemma}

\begin{proof}
First, suppose that $\mu\geq\frac{1}{4}$. Then,
	\begin{align*}
		h_{R,\lambda}(s_\mu) &\geq -\frac{4\sqrt{2}}{\cosh(R)^{\frac{1}{2}\sqrt{\sigma}}}\int_0^R\sqrt{1-\frac{\cosh(u)}{\cosh(R)}} \df u\\
		&\geq -\frac{4\sqrt{2}R}{\cosh(R)^{\frac{1}{2}\sqrt{\sigma}}},
	\end{align*}
For $R$ sufficiently large, this is bounded below by $-1$. The case when $\mu\in [0,\frac{1}{4})$ is trivial since the integrand is non-negative from the spectral parameter $s_\mu$ being purely imaginary. \par 

For the spectral action on $\psi_\lambda$, write $\lambda = \frac{1}{4} - a_\lambda^2$, so that the spectral parameter of $\lambda$ is $s_\lambda = a_\lambda i$. Then by defintion, 
	\begin{align*}
		h_{R,\lambda}(s_\lambda) = \frac{4\sqrt{2}}{\cosh(R)^{\frac{1}{2}\sqrt{\sigma}}}\int_0^R\cosh(a_\lambda u)\sqrt{1-\frac{\cosh(u)}{\cosh(R)}}\df u.
	\end{align*}
By assumption on $\lambda$, it follows that $a_\lambda \in (\sqrt{\sigma},\frac{1}{2})$. Lemma \ref{lem: technical lemma} then shows that
	\begin{align*}
		h_{R,\lambda}(s_\lambda) \geq \frac{4\sqrt{2}}{3}\frac{\sinh(R\sqrt{\sigma})}{\cosh(R)^{\frac{1}{2}\sqrt{\sigma}}},
	\end{align*}
whenever $R\geq 2$. This expression is subsequently bounded below by $\frac{1}{2}e^{\frac{1}{2}\sqrt{\sigma}R}-1$ which is at least $\varepsilon^{-1}$ whenever $R\geq \frac{2}{\sqrt{\sigma}}\log(2+2\varepsilon^{-1})$.
\end{proof}

We now combine the upper bound \eqref{eq: B_R,lambda norm} with Lemma \ref{lem: untemp spectral action} to obtain the desired delocalisation result for small eigenvalues.

\begin{theorem}
\label{thm: det result untemp}
Fix $\varepsilon>0$, and suppose that $X$ is a compact hyperbolic surface with $R=R(X)$ given by condition \eqref{assumption}. Suppose that $\lambda\in (0,\frac{1}{4}-\sigma)$ is an eigenvalue of the Laplacian on $X$, and $\psi_\lambda$ is an $L^2$-normalised eigenfunction with eigenvalue $\lambda$. If $E\subseteq X$ is a measurable set for which
	\begin{align*}
		\|\psi_\lambda\1_E\|_2^2 = \varepsilon,
	\end{align*}
then there exists a universal constant $C>0$, such that for $R$ sufficiently large (dependent only upon $\sigma$ and $\varepsilon$), 
	\begin{align*}
		\mathrm{Vol}(E) \geq \frac{C\varepsilon}{C(X)} e^{(\frac{1}{4}+\frac{1}{2}\sqrt{\sigma})R}.
	\end{align*}
\end{theorem}

\begin{proof}
Suppose that $R$ is sufficiently large, as required by Lemma \ref{lem: untemp spectral action}, so that the bounds for the spectral action of $B_{R,\lambda}$ hold. As in the proof of Theorem \ref{thm: deterministic result}, we can use Lemma \ref{lem: untemp spectral action} to see that 
	\begin{align*}
		\varepsilon^2 \leq \|B_{R,\lambda}\|_{L^1(X)\to L^\infty(X)}\varepsilon \mathrm{Vol}(E).
	\end{align*}
The operator norm of $B_{R,\lambda}$ is controlled as in equation \eqref{eq: B_R,lambda norm} by
	\begin{align*}
		\|B_{R,\lambda}\|_{L^1(X)\to L^\infty(X)} \leq C(X)C_0(\delta) e^{\frac{1}{2}(2\delta -1-\sqrt{\sigma})R},
	\end{align*}
for any $\delta>0$. Set $\delta=\frac{1}{4}$, then we obtain the lower bound
	\begin{align*}
		\mathrm{Vol}(E) \geq \frac{C\varepsilon}{C(X)}e^{(\frac{1}{4}+\frac{1}{2}\sqrt{\sigma})R},
	\end{align*}
where $C = \frac{1}{C_0(\frac{1}{4})}$.
\end{proof}
Combining Theorems \ref{thm: deterministic result} and \ref{thm: det result untemp} then gives the deterministic result Theorem \ref{thm: det thm 1}. Theorem \ref{thm: random result} is then obtained by using Theorem \ref{thm: GLMST19} to probabilistically set $R(X) = c\log(g)$ and $C(X) = \frac{1}{\min\{1,\mathrm{InjRad}(X)\}}$ in Theorem \ref{thm: det result untemp}. 

\section{Acknowledgements}
The author would like to thank Etienne Le Masson and Tuomas Sahlsten for helpful discussions regarding the results presented here. In addition, the author is grateful to Nalini Anantharaman and Laura Monk, as well as Universit\'{e} de Strasbourg for the hospitality during a research visit in February 2020, and for useful comments leading to an improvement on an earlier version of this work. The author also extends thanks to the anonymous referee, whose suggestions improved the presentation of the results.

\bibliographystyle{abbrv}

\end{document}